\documentclass[12pt, a4paper, reqno, oneside]{amsart}
\usepackage{amsmath,amsthm,amssymb,mathrsfs}
\usepackage{color,hyperref,graphicx,comment}
\usepackage[utf8]{inputenc}
\usepackage[all]{xy}
\usepackage{mathtools}
\usepackage{enumitem}
\usepackage{faktor}
\usepackage{tikz-cd}
\usepackage[all]{xy} 
\usepackage{upgreek}
\usepackage{bbold} 
\usepackage{stmaryrd} 

\usepackage{geometry}
\numberwithin{equation}{section}

\theoremstyle{plain}
  \newtheorem{thm}{Theorem}[section]
  \newtheorem{lem}[thm]{Lemma}
  \newtheorem{prop}[thm]{Proposition}
  \newtheorem{cor}[thm]{Corollary}
  \newtheorem*{thm*}{Theorem}  
  \newtheorem*{thmm*}{Main Theorem} 
\theoremstyle{definition}
  \newtheorem{defn}[thm]{Definition}
  \newtheorem{rmk}[thm]{Remark}
  \newtheorem*{ack*}{Acknowledgement}
  \newtheorem*{ref*}{Reference}
  \newtheorem*{ex*}{Example}
  \newtheorem*{ft*}{Fact}
\theoremstyle{plain}

\newcommand\pl{\partial}

\newcommand\kp{\kappa}
\newcommand\ta{\theta}
\newcommand\ld{\lambda}

\newcommand\Gm{\Gamma}
\newcommand\Sm{\Sigma}
\newcommand\dt{\delta}
\newcommand\Dt{\Delta}
\newcommand\vep{\varepsilon}
\newcommand\vph{\varphi}

\newcommand\BR{\mathbb{R}}
\newcommand\BC{\mathbb{C}}
\newcommand\BP{\mathbb{P}}

\newcommand\BN{\mathbb{N}}

\newcommand\BS{\mathbb{S}}

\newcommand\bx{\mathbf{x}}

\newcommand\CC{\mathcal{C}}

\newcommand\CH{\mathcal{H}}
\newcommand\CL{\mathcal{L}}

\title{Self-expanders to the mean curvature flow based on the generalized Lawson-Osserman cone}
\author{Chen-Kuan Lee}
\address{Department of Mathematics\\National Taiwan University\\Taipei 10617\\ Taiwan}
\email{chenkuanlee@ntu.edu.tw}

\begin{document}

\begin{abstract}
We derive the equation of self-similar solutions to mean curvature flow based on the generalized Lawson-Osserman cone and prove the existence of self-expanders by modifying the theory of equilibria in the autonomous system. In particular, those self-expanders are unique if a local assumption is given.
\end{abstract}

\maketitle


\section{\textbf{Introduction}}

Self-similar solutions arise naturally in the study of mean curvature flow (MCF). Self-shrinkers model the behavior of the singularities of MCF \cite{Hui90}. On the other hand, self-expanders are expected to model the behavior of the MCF that emerges from conical singularities \cite{Sta98}. As a result, the analyses of self-similar solutions attract much attention in the past thirty years.

The starting point of this paper is the work of Lawson and Osserman on the minimal graph system in higher codimensions. While the Simons cone is a famous example of codimension one \cite{Sim68}, Lawson and Osserman \cite{LO77} discovered a 4-dimensional minimal cone in $\BR^7$ based on the Hopf fibration: $$\{ (r \bx, \frac{\sqrt{5}}{2} r \CH(\bx)) \big| \, r \in \BR_{\geq 0}, \bx \in \BS^{3} \},$$ where $\CH: \BS^{3} \rightarrow \BS^{2}$ is the Hopf fibration. When equipping $\BR^7$ with the $G_2$-structure, Harvey and Lawson \cite{HL82} proved that it is in fact coassociative, and is therefore area-minimizing. 

Ding and Yuan \cite{DY06} resolved the singularity of the Lawson-Osserman cone. They constructed a family of minimal graphs which are asymptotic to the Lawson-Osserman cone, but are smooth at the origin. Recently, Xu, Yang and Zhang \cite{XYZ19} generalized the Lawson-Osserman cone and the Ding-Yuan's minimal graphs by composing suitable maps from $\BC\BP^n$ to $\BS^m$. Specially, they constructed some maps $\CL: \BS^n \rightarrow \BS^m$ with certain properties  similar as the Hopf fibration.

In this paper, we construct self-expanders to the MCF based on the generalized Lawson-Osserman cone. What follows is our main theorem. The technical terms will be explained in section 2.
\begin{thmm*}
Suppose that $\CL: \BS^n \rightarrow \BS^m$ is an LOMSE of $(n, p, k)$-type, where $(n, p, k) = (3, 2, 2), (5, 4, 2), (5, 4, 4)$ or $n \geq 7$. Then, there exist positive constants $\vep_0$, $R_0$ and $\vph_0$ with the following significance: for any $\vep \in (0, \vep_0]$ and $R \in (0, R_0]$, there exists a smooth self-expander in $\BR^{n + m + 2}$ of the form $$\Sm = \{(r \bx, f(r) \CL(\bx)) \, \big| \, r \in \BR_{\geq 0}, \bx \in \BS^{n}\}$$ such that $f: \BR_{\geq 0} \rightarrow \BR$ satisfies $$f(R) = \vep R $$ Moreover, $f$ has the following properties:
\begin{enumerate}
    \item \label{result1} $0 \leq f \leq \vph_0 r$ and $0 \leq f_r.$
    \item \label{result2} $f \in O(r^k)$ and $f_r \in O(r^{k - 1})$ as $r \rightarrow 0$.
    \item Assume that \begin{equation} \label{cond1}
        \left\{\begin{aligned}
            0 \leq f &\leq \vep r \\
            0 \leq f_r &\leq (2k -1) \frac{f}{r}
        \end{aligned}\right. \quad \text{for all } r \in (0, R],
    \end{equation} then $f$ is unique.
    \item $\lim_{r \rightarrow \infty} \frac{f(r)}{r}$ exists. In other words, $\Sm$ is asymptotic to the cone $$\{ ( r \bx, r \vph_{\infty} \CL(\bx)) \, \big| \, r \in \BR_{\geq 0}, \bx \in \BS^{n}\},$$ where $\vph_{\infty} = \lim_{r \rightarrow \infty} \frac{f(r)}{r}$, as $r \rightarrow \infty$.
\end{enumerate}
\end{thmm*}

The organization of this paper is as follows. In section 3, we use symmetry to transform the equation of self-similar solutions to a second order ODE. Section 4 contains a stable curve theorem, which is a generalization of the theory of equilibria in the autonomous system. After an appropriate change of variables and using the stable curve theorem, the existence of self-expanders is obtained in section 5. The uniqueness under the assumption \eqref{cond1} is established in section 6. Finally, we analyze the asymptotic behavior in section 7.

\begin{ack*} The author is really grateful to Prof. Chung-Jun Tsai for his helpful and inspiring comments. Part of this paper is from the author's master thesis. He also appreciates Prof. Mao-Pei Tsui's suggestions on further generalizations and Chin-Bin Hsu’s indications of deficiencies in an earlier draft. The author thanks the anonymous referees for their helpful comments that improved the quality of this paper.
\end{ack*}


\section{\textbf{Preliminary}}

\subsection{Self-similar solutions to mean curvature flow}

We recall some background material of mean curvature flow (MCF).

\begin{defn}
Let $\Sigma$ be a smooth submanifold in a Riemannian manifold $M$. If there exists a family of smooth immersions $F_t: \Sigma \rightarrow M$ satisfying \begin{equation}
    \left\{\begin{aligned}
    &(\frac{\pl F_t}{\pl t} (\bx))^{\perp} = H_{\Sigma_t}(\bx) \\
    &F_{0}(\bx) = \bx \text{ for all } \bx \in \Sm
    \end{aligned}\right.,
\end{equation} where $(\frac{\pl F_t}{\pl t})^{\perp}$ denotes the projection of the $\frac{\pl F_t}{\pl t}$ to the normal bundle of $\Sigma$, $$H_{\Sigma_t} \coloneqq (g_t)^{ij} \nabla^\perp_{\frac{\pl F_t}{\pl x^i}} \frac{\pl F_t}{\pl x^j}$$ denotes the mean curvature vector of $\Sigma_t$, and $$(g_t)_{ij} \coloneqq \langle \frac{\pl F_t}{\pl x^i}, \frac{\pl F_t}{\pl x^j} \rangle_M,$$ then $F_t$ is called the mean curvature flow (MCF) of $\Sigma$.
\end{defn}

In geometric flows, singularities are often modeled on soliton solutions. For MCF, there are two types of soliton solutions in Euclidean space that are of particular interest: the one moving by scaling, and the other one moving by translation. In this paper, we focus on the first one.

\begin{defn}
A submanifold in Euclidean space, $F: \Sm \rightarrow \BR^n$, is called a self-similar solution to the MCF if \begin{equation}
    H_{\Sigma} \equiv C F^\perp
\end{equation} for some constant $C \in \BR$, where $H_\Sigma$ is the mean curvature vector of $F: \Sm \rightarrow \BR^n$ and $F^{\perp}$ denotes the projection of the position vector $F$ in $\BR^n$ to the normal bundle of $\Sigma$. Moreover, it is called a self-shrinker if $C < 0$ and a self-expander if $C> 0$. When $C =0$, the submanifold is minimal.
\end{defn}

Notice that if $\Sigma$ is a self-similar solution, then $F_t$ defined by $$F_t = \sqrt{1 + 2Ct} F $$ moves by the MCF.

\begin{rmk}
After rescaling, only the sign of $C$ matters. It suffices to consider $C = 1, 0, -1$.
\end{rmk}

\subsection{The generalized Lawson-Osserman cone}

We first recall the definition of the Hopf maps $$\CH_d: \BS^{2d - 1} \rightarrow \BS^{d},$$ where $d = 2, 4, 8$.

\begin{defn}
We identify $\BR^{d}$ with the normed algebra: $\BC$, $\mathbb{H}$, $\mathbb{O}$ for $d = 2, 4, 8$, respectively. The Hopf map is defined by $$\begin{matrix}
\CH_d: & \BS^{2d - 1} & \rightarrow & \BS^{d}\\
 & (p, q) & \mapsto & (\|p\|^2 - \|q\|^2, 2q\overline{p})
\end{matrix}.$$
\end{defn}

By using the Hopf maps, the Lawson-Osserman cones are as follows.

\begin{defn}
The Lawson-Osserman cones are $$C_d = \{ (r \bx, \kp_d \, r \, \CH_d(\bx)) \big| \, r \in \BR_{\geq 0}, \bx \in \BS^{2d-1} \},$$ where $d = 2, 4, 8$, $\kp_d = \sqrt{\frac{2d+1}{4(d-1)}}$ and $\CH_d: \BS^{2d-1} \rightarrow \BS^{d}$ is the Hopf map. The constant $\kp_d$ is the unique one such that the cone is minimal.
\end{defn}

We now explain the notion of generalized Lawson-Osserman cones introduced by Xu, Yang and Zhang \cite{XYZ19}.

\begin{defn} \label{LOMSE}
For a smooth map $\CL: \BS^n \rightarrow \BS^m$, if there exists an acute angle $\ta$ such that $$M_{\CL, \ta} \coloneqq \{( \bx \cos \ta, \CL(\bx) \sin \ta) \, \big| \, \bx \in \BS^n\}$$ is a minimal submanifold of $\BS^{n+m+1}$, then $\CL$ is called a Lawson-Osserman map (LOM), $M_{\CL, \ta}$ is called the associated Lawson-Osserman sphere (LOS), and the cone $C_{\CL, \ta}$ over $M_{\CL, \ta}$ is called the corresponding Lawson-Osserman cone (LOC).

Suppose that $\CL: \BS^n \rightarrow \BS^m$ is an LOM whose image is not totally geodesic. Then $\CL$ is called an LOMSE if all non-zero singular values\footnote{Here, a singular value $\ld(\bx)$ of $(\CL_*)_\bx$ means the square root of an eigenvalue of the self-adjoint operator $((\CL_*)_\bx)^* (\CL_*)_\bx$.} $\ld(\bx)$ of $(\CL_*)_\bx$ are equal for every $\bx \in \BS^n$. Note that $\ld(\bx)$ is a continuous function of $\bx$. For an LOMSE, one can show that $\ld(\bx)$ equals a constant $\ld$ and $\CL$ has constant rank $p$. Furthermore, all components of this vector-valued function $\CL = (\CL_1, \cdots, \CL_{m+1}) \in \BS^m \subset \BR^{m+1}$ are spherical harmonic functions of degree $k \geq 2$ and the singular value $\ld = \sqrt{\frac{k(n + k - 1)}{p}}$ (cf. \cite[Theorem 2.8]{XYZ19}). Such an $\CL$ is called an LOMSE of $(n, p, k)$-type. 
\end{defn}

It can easily be seen that the Hopf map $\CH_d$ is exactly the LOMSE of $(2d-1, d, 2)$-type. Here is the classification and all are based on the Hopf maps.

\begin{prop}{\cite[Theorem 2.10]{XYZ19}}
There are three families of LOMSEs (see Definition \ref{LOMSE}) of $(n, p, k)$-type that generalize the Hopf maps:
\begin{enumerate}
    \item $(n, p) = (2l+1, 2l)$ and $k = 2 q$, $l, q \in \BN$, which generalize the Hopf map $\CH_2$.
    \item $(n, p) = (4l+3, 4l)$ and $k = 2 q$, $l, q \in \BN$, which generalize the Hopf map $\CH_4$.
    \item $(n, p) = (15, 8)$ and $k = 2 q$, $q \in \BN$, which generalize the Hopf map $\CH_8$.
\end{enumerate}
\end{prop}

On the other hand, there are minimal graphs that resolve the singularities of LOCs.

\begin{prop}{\cite[Theorem 3.5]{XYZ19}}
Suppose that $\CL$ is an LOMSE (see Definition \ref{LOMSE}) of $(n, p, k)$-type. Then there is an analytic entire minimal graph of the form $$M_{\CL, \rho} = \{(r \bx, \rho(r) \CL(\bx)) \big| \, r \in \BR_{\geq 0}, \bx \in \BS^{n}\}$$ that is asymptotic to $C_{\CL, \ta}$. In particular, when $(n, p, k) = (3, 2, 2), (5, 4, 2), (5, 4, 4)$ or $n \geq 7$, $M_{\CL, \rho}$ is below $C_{\CL, \ta}$ in the sense that $$\frac{\rho}{r} \leq \vph_0 \coloneqq \tan \ta = \sqrt{\frac{p \ld^2 - n}{(n-p) \ld^2}}$$ for all $r$.
\end{prop}

\subsection{Harmonic maps}

Let $(M^m, g_M)$ and $(N^n, g_N)$ be Riemannian manifolds. Let $f: M \rightarrow N$ be a smooth map. Then the energy density of $f$ at $x \in M$ is defined to be $$e_{g_M}(f) \coloneqq \frac{1}{2} \sum_{i = i}^m g_N(f_* e_i, f_* e_i),$$ where $\{e_i\}_{i =1}^m$ is an orthonormal basis of $T_x M$. Since we will change the metric on the domain, $M$, we record the metric of the domain. The energy of $f$ is defined to be $$E_{g_M}(f) \coloneqq \int_M e_{g_M}(f) \, dvol_{g_M}.$$

Let $\nabla^{TM}$ be the Levi-Civita connection on $TM$ with respect to $g_M$. Let $\nabla^{f^*TN}$ be the connection on $f^*TN$ that is compatible with $g_N$. Then the second fundamental form of $f$ is defined to be $$B_{XY}(f) \coloneqq \nabla^{f^*TN}_X Y - f_*(\nabla^{TM}_X Y).$$ If $B(f) \equiv 0$, then $f$ is called a totally geodesic map. The tension field of $f$ is defined to be the trace of its second fundamental form under $g_M$: $$\tau_{g_M}(f) \coloneqq \sum_{i = 1}^m B_{e_i e_i}(f).$$ If $\tau_{g_M}(f) \equiv 0$, then $f$ is called a harmonic map. Using the first variation, it can be showed that $f$ is harmonic if and only if $f$ is a critical point of energy functional $E(\cdot)$ (cf. \cite[Chap. 1.2.3, p.13--14]{Xin96}).

Given a smooth function $f: (M, g_M) \rightarrow (\BR^n, g_{\text{std}})$, its tension field has a simple expression, which is \begin{equation} \label{lp}
    \tau_{g_{M}}(f) = \Dt_{g_{M}} f,
\end{equation} where $\Dt_{g_M}$ denotes the Laplace-Beltrami operator. Namely, the harmonicity here matches the usual one.

Given an isometric immersion $\iota: (M, g_M) \hookrightarrow (N, g_N)$, the second fundamental form of $\iota$ is just the second fundamental form of $M$ as an immersed submanifold of $N$ (cf. \cite[Chap. 1.2.4, p.15]{Xin96}). Moreover, it implies that the tension field of $f$ is exactly the mean curvature vector of $M$ in $N$. In other words, That $M$ is minimal in $N$ is equivalent to $\iota$ is harmonic in this case.

There is a composition formula for tension fields. Let $M, N, \Bar{N}$ be Riemannian manifolds. Suppose that $f: M \rightarrow N$ and $\Bar{f}: N \rightarrow \Bar{N}$ are smooth maps. Then $$\tau_{g_M}(\Bar{f} \circ f) = \Bar{f}_*(\tau_{g_M}(f)) + \sum_{i = 1}^m B_{f_* e_i, f_* e_i}(\Bar{f})$$ (cf. \cite[Chap. 1.4.1, p. 28]{Xin96}). Suppose that $\Bar{f}$ is an isometric immersion. Combining the discussion in the previous paragraphs, we simply write the formula as \begin{equation} \label{compo}
    \tau_{g_M}(\Bar{f} \circ f) = \tau_{g_M}(f) + \sum_{i = 1}^m h(f_* e_i, f_* e_i),
\end{equation} where $h$ is the second fundamental form of $N$ in $\Bar{N}$.


\section{\textbf{Necessary and sufficient conditions for graphical self-similar solutions}}

Given a smooth function $f: \BR_{\geq 0} \rightarrow \BR$ and a smooth map $\CL: \BS^n \rightarrow \BS^m$, we consider \begin{equation}
    \begin{matrix} 
    F: & \BR_{\geq 0} \times \BS^n & \rightarrow & \BR^{n+1} \times \BR^{m+1} \\
    \, & (r, \bx) & \mapsto & (r \bx, f(r) \CL(\bx))
    \end{matrix}
\end{equation} and study when $\Sm = F(\BR_{\geq 0} \times \BS^n)$ is a self-similar solution to the MCF in $\BR^{n+m+2}$.

Let $g$ be the metric on $\Sm$ induced by the standard Euclidean metric on $\BR^{n+m+2}$ and $g_r = I^*_r g$, where \begin{equation} \label{I_r}
    \begin{matrix}
I_r: & \BS^n & \rightarrow & \Sm \\
\, & \bx & \mapsto & (r \bx, f(r) \CL(\bx))
\end{matrix}.
\end{equation} Let $\nabla$ be the Levi-Civita connection on $(\Sm, g)$. Note that $\nabla_v r = \nabla_v f = 0$ for all $v \in T\BS^n$.

\begin{thm}
\label{thm0}
Assume that $f > 0$ whenever $r > 0$. Then $\Sm$ is a self-similar solution to the MCF, i.e. $H_\Sm = C F^\perp$ where $C = \pm 1$, if and only if the following two conditions hold:
\begin{enumerate}
    \item For each $r \in \BR_{> 0}$, $\CL: (\BS^n, g_r) \rightarrow (\BS^m, g_m)$ is harmonic, where $g_r = I^*_r g$ (see \eqref{I_r}), and $g_m$ is the standard metric on $\BS^m$.
    \item For each $r \in \BR_{> 0}$, \begin{equation} \label{eq}
        \Dt_g f - 2e_{g_r}(\CL)f + \frac{C( r f_r - f)}{1 + f_r^2} = 0
    \end{equation} in $I_r(\BS^n)$, where $e_{g_r}(\CL)$ is the energy density of $\CL: (\BS^n, g_r) \rightarrow (\BS^m, g_m)$.
\end{enumerate}

Moreover, suppose that $\ld_1, \cdots, \ld_n$ are singular values of $(\CL_*)_{\bx}: (T_{\bx}\BS^n, g_n) \rightarrow (T_{\CL(\bx)}\BS^m, g_m)$. Then condition (2) is equivalent to the equation $$\frac{f_{rr}}{1 + f_r^2} + \sum_{i = 1}^n \frac{r f_r - \ld_i^2 f}{r^2 + \ld_i^2 f ^2} + C (r f_r - f) = 0.$$
\end{thm}

\begin{proof}
For any $p \in \BN$, let $\iota_p: (\BS^p, g_p) \rightarrow (\BR^{p+1}, g_{\text{std}})$ be the natural isometric immersion. Throughout the proof, $\langle \cdot, \cdot \rangle$ denotes the standard Euclidean inner product. First, we write down $F$ carefully:  \begin{equation} \label{F}
    F(r, \bx) = (r X_1(\bx), f(r) X_2(\bx)) \in \BR^{n+1} \times \BR^{m+1},
\end{equation} where \begin{equation} \label{X1}
    X_1 = \iota_n \circ \text{Id}_{\BS^n}: (\BS^n, g_r) \rightarrow (\BR^{n+1}, g_{\text{std}})
\end{equation} and \begin{equation} \label{X2}
    X_2 = \iota_m \circ \CL: (\BS^n, g_r) \rightarrow (\BR^{m+1}, g_{\text{std}}).
\end{equation}

From now on, all the calculations are performed at a fixed point $(r, \bx) \in \BR_{> 0} \times \BS^n$. Let $\{v_i\}_{i = 1}^{n}$ be an SVD-basis of $(T_{\bx}\BS^n, g_n)$ with respect to $(\CL_*)_{\bx}$ such that each $v_i$ subjects to the singular value $\ld_i$. It follows that \begin{equation} \label{frame}
    \begin{aligned} \frac{\pl}{\pl r} &\coloneqq (X_1, f_r X_2) \\
    V_i &\coloneqq (F_*)_{(r, \bx)}(v_i) = (r v_i, f \cdot (\CL_*)_\bx v_i) \quad \forall \, 1 \leq i \leq n
    \end{aligned}
\end{equation} form an induced orthogonal basis of $(T_{(r, \bx)}\Sm, g).$ It is a straightforward computation to find that \begin{equation} \label{metric}
    \begin{aligned} g_{00} &\coloneqq \langle \frac{\pl}{\pl r}, \frac{\pl}{\pl r} \rangle = 1 + f_r^2 \\ 
    g_{0i} &\coloneqq \langle \frac{\pl}{\pl r}, V_i \rangle = 0 \\ 
    g_{ij} &\coloneqq \langle V_i, V_j \rangle = (r^2 + \ld_i^2 f^2) \dt_{ij} \quad \forall \, 1 \leq i, j \leq n \end{aligned}
\end{equation}
 and \begin{equation}
     \label{bracket}\nabla_{\frac{\pl}{\pl r}} V_i - \nabla_{V_i} \frac{\pl}{\pl r} = [\frac{\pl}{\pl r}, V_i] = 0.
 \end{equation} Here, we slightly abuse the notation for $\frac{\pl}{\pl r}$ and $V_i$'s. They also denote the extended local vector field near $(r, \bx)$. It is not hard to see that an extension satisfying \eqref{bracket} exists.

By \eqref{F}, \eqref{frame} and \eqref{metric} $F^\perp$ is equal to \begin{equation} \label{F^perp}
    \begin{aligned} F^\perp &= F - F^T \\
    &= F - g^{00} \langle F, \frac{\pl}{\pl r}\rangle \frac{\pl}{\pl r} - \sum_{i = 1}^n g^{ii} \langle F, V_i \rangle V_i \\
    &= (r X_1, f X_2) - \frac{r + f_rf}{1+f_r^2} (X_1, f_r X_2) \\
    &= \frac{- rf_r + f}{1+f_r^2} (-f_r X_1, X_2).
    \end{aligned}
\end{equation}

For mean curvature vector $H_\Sm$, note that \begin{equation} \label{mcv}
    H_\Sm = \Dt_g F = (\Dt_g(r X_1), \Dt_g(f X_2)).
\end{equation} We focus on the second component. Since $f$ is independent of $\bx$ and $X_2$ is independent of $r$, \eqref{metric} implies that \begin{equation} \label{Dt_g(f X_2)}
    \begin{aligned} \Dt_g(f X_2) &= (\Dt_g f) X_2 + f (\Dt_g X_2) + 2 ( g^{00} f_r \nabla_{\frac{\pl}{\pl r}} X_2 + \sum_{i=1}^{n} g^{ii} V_i(f) \nabla_{V_i} X_2) \\
    &= (\Dt_g f) X_2 + f (\Dt_g X_2).\end{aligned}
\end{equation}
Moreover, in the Riemannian submanifold $(\BS^n, g_r)$ of $(\Sm, g)$, we use \eqref{metric} to get \begin{equation} \label{Dt_g X_2}
    \begin{aligned}\Dt_g X_2 &= \Dt_{g_r} X_2 + g^{00} (\nabla_{\frac{\pl}{\pl r}} \nabla_{\frac{\pl}{\pl r}} X_2 - \nabla_{\nabla_{\frac{\pl}{\pl r}} \frac{\pl}{\pl r}} X_2) \\
    &= \Dt_{g_r} X_2 - g^{00} \nabla_{\nabla_{\frac{\pl}{\pl r}} \frac{\pl}{\pl r}} X_2 \\
    &= \Dt_{g_r} X_2 - g^{00} \sum_{i =1}^n g^{ii} \langle \nabla_{\frac{\pl}{\pl r}} \frac{\pl}{\pl r}, V_i \rangle \nabla_{V_i} X_2 \\
    &= \Dt_{g_r} X_2 + g^{00} \sum_{i =1}^n g^{ii} \langle \frac{\pl}{\pl r}, \nabla_{\frac{\pl}{\pl r}} V_i \rangle \nabla_{V_i} X_2 \\
    &= \Dt_{g_r} X_2 + g^{00} \sum_{i =1}^n g^{ii} \langle \frac{\pl}{\pl r}, \nabla_{V_i} \frac{\pl}{\pl r} \rangle \nabla_{V_i} X_2 \\
    &= \Dt_{g_r} X_2 + g^{00} \sum_{i =1}^n g^{ii} \frac{V_i(g_{00})}{2} \nabla_{V_i} X_2 \\
    &= \Dt_{g_r} X_2. \end{aligned}
\end{equation}
Recall that $\CL: (\BS^n, g_r) \rightarrow (\BS^m, g_m)$. According to  \eqref{X2}, \eqref{lp} and \eqref{compo}, we have \begin{equation} \label{Dt_(g_r)X_2}
    \Dt_{g_r} X_2 = \tau_{g_r}(X_2) = \tau_{g_r}(\iota_m \circ \CL) = \tau_{g_r}(\CL) - 2 e_{g_r}(\CL) X_2.
\end{equation} Therefore, \eqref{Dt_g(f X_2)}, \eqref{Dt_g X_2} and \eqref{Dt_(g_r)X_2} imply that \begin{equation} \label{sim}
    \Dt_g(f X_2) = f \cdot \tau_{g_r}(\CL) + (\Dt_g f - 2 e_{g_r}(\CL) f) X_2.
\end{equation}

Notice that $\tau_{g_r}(\CL) \in \Gm(\CL^*T\BS^m)$. In other words, $\langle \tau_{g_r}(\CL), X_2 \rangle = 0$. It follows from $H_\Sm = C F^\perp$, \eqref{F^perp} and \eqref{mcv} that $$\tau_{g_r}(\CL) = 0 \text{ and \eqref{eq} holds true}.$$

Conversely, if $\tau_{g_r}(\CL) = 0$ and \eqref{eq} holds true, it follows from \eqref{mcv} that \begin{equation} \label{mcv2}
    H_\Sm = (\Dt_g(r X_1), \frac{C(- rf_r + f)}{1+f_r^2} X_2).
\end{equation} We also recall that $H_\Sm \in N \Sm$. By using \eqref{frame} and \eqref{metric}, we have $$\begin{aligned} 0 &= \langle \frac{\pl}{\pl r}, H_\Sm \rangle \\
&= \langle (X_1, f_r X_2), (\Dt_g(r X_1), \frac{C(- rf_r + f)}{1+f_r^2} X_2) \rangle \\
&= \langle X_1, \Dt_g(r X_1) \rangle + f_r \cdot \frac{C(- rf_r + f)}{1+f_r^2}\end{aligned}$$ and $$\begin{aligned} 0 = \langle V_i, H_\Sm \rangle = \langle (r v_i, f \cdot (\CL_*)_\bx v_i), (\Dt_g(r X_1), \frac{C(- rf_r + f)}{1+f_r^2} X_2) \rangle = r \langle v_i, \Dt_g(r X_1) \rangle \end{aligned}.$$ Therefore, $\Dt_g(r X_1) = - f_r \cdot \frac{C(- rf_r + f)}{1+f_r^2} X_1 $; with \eqref{F^perp} and \eqref{mcv2}, $$H_\Sm = \frac{C(- rf_r + f)}{1+f_r^2} (- f_r X_1, X_2) = C F^\perp,$$ which means that $\Sm$ is a self-similar solution to the MCF.

Finally, we use \eqref{metric} to derive an explicit expression of $\Dt_g f - 2 e_{g_r}(\CL) f$. 
    $$\begin{aligned} \Dt_g f =& f_r \Dt_g r + f_{rr} |\text{grad}_g f|^2 \\
    =& f_r (g^{00} (\nabla_{\frac{\pl}{\pl r}} \nabla_{\frac{\pl}{\pl r}} r - \nabla_{\nabla_{\frac{\pl}{\pl r}} \frac{\pl}{\pl r}} r) + \sum_{i=1}^n g^{ii} (\nabla_{V_i} \nabla_{V_i} r - \nabla_{\nabla_{V_i} V_i} r)) + f_{rr} g^{00} \\
    =& f_r( -g^{00}\nabla_{\nabla_{\frac{\pl}{\pl r}} \frac{\pl}{\pl r}} r - \sum_{i=1}^n g^{ii} \nabla_{\nabla_{V_i} V_i} r) + \frac{f_{rr}}{1 + f_r^2} \\
    =& f_r ( - (g^{00})^2 \langle \nabla_{\frac{\pl}{\pl r}} \frac{\pl}{\pl r}, \frac{\pl}{\pl r} \rangle - \sum_{i =1}^n g^{ii} g^{00} \langle \nabla_{V_i} V_i, \frac{\pl}{\pl r} \rangle) + \frac{f_{rr}}{1 + f_r^2} \\
    =& f_r ( - \frac{\nabla_{\frac{\pl}{\pl r}} \langle \frac{\pl}{\pl r}, \frac{\pl}{\pl r} \rangle}{2 (1 + f_r^2)^2} + \sum_{i = 1}^n \frac{\nabla_{\frac{\pl}{\pl r}} \langle V_i, V_i \rangle}{2 (1 + f_r^2)(r^2 + \ld_i^2 f^2)}) + \frac{f_{rr}}{1 + f_r^2} \\
    =& \frac{f_{rr}}{(1 + f_r^2)^2} + \sum_{i =1}^n \frac{r f_r + \ld_i^2 f_r^2 f }{(1 + f_r^2)(r^2 + \ld_i^2 f^2)}\end{aligned}$$
 and 
    $$2 e_{g_r}(\CL) = \sum_{i = 1}^n \frac{g_n(\CL_*(v_i), \CL_*(v_i))}{g_r(v_i, v_i)} = \sum_{i = 1}^n \frac{\langle \CL_*(v_i), \CL_*(v_i) \rangle}{\langle V_i, V_i \rangle} = \sum_{i = 1}^n \frac{\ld_i^2}{r^2 + \ld^2_i f^2}.$$ 
It follows that \eqref{eq} is equivalent to $$\frac{f_{rr}}{1 + f_r^2} + \sum_{i =1}^n \frac{r f_r - \ld_i^2 f}{r^2 + \ld_i^2 f^2} + C(r f_r - f) = 0.$$ It finishes the proof of this theorem.
\end{proof}

Now, we consider $\CL$ to be an LOMSE of $(n, p, k)-$type and obtain a simple version of Theorem \ref{thm0}. It has been proved that LOMSEs automatically satisfy condition (1) (cf. \cite[Sec. 3.2]{XYZ19}). Moreover, for an LOMSE of $(n, p, k)-$type, it has only two constant singular values $$\ld = \sqrt{\frac{k(n+k-1)}{p}}$$ and 0 of constant multiplicities $p$ and $n-p$ respectively (cf. \cite[Theorem 2.8]{XYZ19}). Therefore, we have the following simple version.

\begin{cor}
Assume that $f > 0$ whenever $r > 0$ and that $\CL$ is an LOMSE of $(n, p, k)$-type. Then $\Sm$ is a self-similar solution to the MCF if and only if $f$ satisfies \begin{equation}\label{eq1}
    \frac{f_{rr}}{1 + f_r^2} + \frac{(n-p)f_r}{r} + \frac{p(r f_r - \ld^2 f)}{r^2 + \ld^2 f^2} + C (r f_r - f) = 0,
\end{equation} where $\ld = \sqrt{\frac{k(n+k-1)}{p}}$.
\end{cor}

\begin{rmk}
Suppose that $C = 0$. Then $$\frac{f_{rr}}{1 + f_r^2} + \frac{(n-p)f_r}{r} + \frac{p(r f_r - \ld^2 f)}{r^2 + \ld^2 f^2} = 0$$ is a necessary and sufficient condition for $\Sm$ to be a minimal graph (cf. \cite[Theorem 3.2]{XYZ19}). Moreover, for $(n, p, k) = (3, 2, 2), (7, 4, 2), (15, 8, 2)$, the equation was first found by Ding and Yuan (cf. \cite[Sec. 2]{DY06}).
\end{rmk}


\section{\textbf{A stable curve theorem for a non-autonomous planar system}}

The following section is essentially based on the author's master thesis (cf. \cite[Sec. 4]{Lee22}). For the sake of completeness, we provide the details here.

Throughout this section, we consider the following system of ODEs: \begin{equation}\label{eq2}\left\{\begin{aligned} X_t(t) &= -\kp X(t) + f_1(X(t), Y(t)) + e^{-t} g_1(X(t), Y(t)) \\
Y_t(t) &= \mu Y(t) + f_2(X(t), Y(t)) + e^{-t} g_2(X(t), Y(t)) \end{aligned}\right.,\end{equation} where $\mu > 0 > - \kp$, $\frac{f_i(X, Y)}{\sqrt{X^2 + Y^2}} \rightarrow 0$ as $(X, Y) \rightarrow (0, 0)$ and $g_i \in O(\sqrt{X^2 + Y^2})$ as $(X, Y) \rightarrow (0, 0)$ $\forall \, i = 1, 2$. In this case, $(0, 0)$ is an equilibrium.

If we omit the exponential terms, then it becomes a classical planar autonomous system. Under that situation, $(0, 0)$ is in fact a saddle equilibrium. There is a stable curve theorem for such case (cf. \cite[Chap. 8.3, p.169]{HSD13}), which states that we can find an $\vep > 0$ and a unique \emph{local} stable curve of the form $Y = h(X)$ that is defined for $|X| < \vep$ and satisfies $h(0) = 0$. Moreover, this curve is tangent to the $X$-axis and all solutions with initial conditions that lies on this curve tend to the origin as $t \rightarrow \infty$.

This section aims to provide a similar stable curve theorem for the system \eqref{eq2}. We first give some notations to clarify the meaning of \emph{local}. Let $S_\vep$ be the square bounded by $\{X = \pm \vep\}$ and $\{Y = \pm \vep\}$. We also define \begin{equation}
    R_{M_1, M_2, \vep}^+ \coloneqq \{(X, Y) \in S_\vep \, \big| \, - M_2 X \leq Y \leq M_1 X, X \geq 0\}
\end{equation} and \begin{equation}
    E_{M_1, M_2, \vep}^+ \coloneqq R_{M_1, M_2, \vep}^+ \cap \{X = \vep\}.
\end{equation} Then $E_{M_1, M_2, \vep}^+$ is a part of the boundary of $R_{M_1, M_2, \vep}^+$.

Now, we state two lemmas about the behavior of the vector field inside $R_{M_1, M_2, \vep}^+$.

\begin{lem}
\label{lem1}
Given any $M_1, M_2 > 0$ and $\dt \in (0, \kp)$, there exists $\vep > 0$ and $T > 0$ such that $X_t < - (\kp - \dt) X < 0$ in $R_{M_1, M_2, \vep}^+$ whenever $t \geq T$.
\end{lem}

\begin{proof}
Let $M \coloneqq \max\{M_1, M_2\}$. Since $\frac{f_1(X, Y)}{\sqrt{X^2 + Y^2}} \rightarrow 0$ as $(X, Y) \rightarrow (0, 0)$, we may choose $\vep_1 > 0$ so that $$|f_1(X, Y)| \leq \frac{\dt}{2 \sqrt{M^2+1}} \sqrt{X^2 + Y^2} \quad \forall \, (X, Y) \in S_{\vep_1}.$$
Moreover, since $g_1 \in O(\sqrt{X^2 + Y^2})$, there exist $\vep_2 > 0$ such that $$|g_1(X, Y)| < \tilde{C} \sqrt{X^2 + Y^2} \quad \forall \, (X, Y) \in S_{\vep_2}$$ for some positive constant $\tilde{C}$. Let $\vep = \min\{\vep_1, \vep_2\}$. Now, we set $T > 0$ so that $$e^{-t} < \frac{\dt}{2 \tilde{C} \sqrt{M^2+1}} \quad \forall \, t > T.$$ Note that in $R_{M_1, M_2, \vep}^+$, $|Y| \leq M X$, and thus $\sqrt{X^2 + Y^2} \leq \sqrt{M^2+1} X$. It follows that $$\begin{aligned}
X_t &= -\kp X + f_1(X, Y) + e^{-t} g_1(X, Y) \\
&\leq -\kp X + |f_1(X, Y)| + e^{-t} |g_1(X, Y)| \\
&\leq -\kp X + \frac{\dt}{2 \sqrt{M^2+1}} \sqrt{X^2 + Y^2} + \frac{\dt}{2 \tilde{C} \sqrt{M^2+1}} (\tilde{C} \sqrt{X^2 + Y^2}) \\
&\leq -\kp X + \frac{\dt}{2 \sqrt{M^2+1}} \sqrt{M^2+1} X + \frac{\dt}{2 \tilde{C} \sqrt{M^2+1}} (\tilde{C} \sqrt{M^2+1} X) \\
&= - (\kp - \dt) X < 0
\end{aligned}$$ whenever $t \geq T$. It finishes the proof of this lemma.
\end{proof}

\begin{lem}
\label{lem2}
Given any $M_1, M_2 > 0$, there exists $\vep > 0$ and $T > 0$ such that $Y_t > 0$ on $\{(X,Y) \in R_{M_1, M_2, \vep}^+ \, \big| \, Y = M_1 X, X > 0 \}$ and $Y_t < 0$ on $\{(X,Y) \in R_{M_1, M_2, \vep}^+ \, \big| \, Y = - M_2 X, X > 0 \}$ whenever $t \geq T$.
\end{lem}

\begin{proof}
Let $M \coloneqq \max\{M_1, M_2\}$ and $m \coloneqq \min\{M_1, M_2\}$. Since $\frac{f_2(X, Y)}{\sqrt{X^2 + Y^2}} \rightarrow 0$ as $(X, Y) \rightarrow (0, 0)$, we may choose $\vep_1 > 0$ so that $$|f_2(X, Y)| \leq \frac{m \mu}{3 \sqrt{M^2+1}} \sqrt{X^2 + Y^2} \quad \forall \, (X, Y) \in S_{\vep_1}.$$
Furthermore, since $g_2 \in O(\sqrt{X^2 + Y^2})$, there exist $\vep_2 > 0$ such that $$|g_2(X, Y)| < \tilde{C} \sqrt{X^2 + Y^2} \quad \forall \, (X, Y) \in S_{\vep_2}$$ for some positive constant $\tilde{C}$. Let $\vep = \min\{\vep_1, \vep_2\}$ and choose $T > 0$ so that $$e^{-t} < \frac{m \mu}{3 \tilde{C} \sqrt{M^2+1}} \quad \forall \, t > T.$$ Then on $\{(X,Y) \in R_{M_1, M_2, \vep}^+ \, \big| \, Y = M_1 X, X > 0 \}$, $$\begin{aligned}
Y_t &= \mu Y + f_2(X, Y) + e^{-t} g_2(X, Y) \\
&\geq \mu Y - |f_2(X, Y)| - e^{-t} |g_2(X, Y)| \\
&\geq \mu Y - \frac{m \mu}{3 \sqrt{M^2+1}} \sqrt{X^2 + Y^2} - \frac{m \mu}{3 \tilde{C} \sqrt{M^2+1}} (\tilde{C} \sqrt{X^2 + Y^2}) \\
&\geq \mu Y - \frac{m \mu}{3 \sqrt{M^2+1}} \sqrt{M_1^{-2}+1} Y - \frac{m \mu}{3 \tilde{C} \sqrt{M^2+1}} (\tilde{C} \sqrt{M_1^{-2}+1} Y) \\
&\geq \frac{\mu}{3} Y > 0
\end{aligned}$$ whenever $t \geq T$. Similarly, on $\{(X,Y) \in R_{M_1, M_2, \vep}^+ \, \big| \, Y = - M_2 X, X > 0 \}$, $$\begin{aligned}
Y_t &= \mu Y + f_2(X, Y) + e^{-t} g_2(X, Y) \\
&\leq \mu Y + |f_2(X, Y)| + e^{-t} |g_2(X, Y)| \\
&\leq \mu Y - \frac{m \mu}{3 \sqrt{M^2+1}} \sqrt{M_2^{-2}+1} Y - \frac{m \mu}{3 \tilde{C} \sqrt{M^2+1}} (\tilde{C} \sqrt{M_2^{-2}+1} Y) \\
&\leq \frac{\mu}{3} Y < 0
\end{aligned}$$ whenever $t \geq T$. It finishes the proof of this lemma.
\end{proof}

Using the above lemmas, we are able to show the existence of the stable curve for the system \eqref{eq2}; however, unlike the classical one, it depends on the initial time.

\begin{thm}
\label{thm1}
Given a system of ODEs $$\left\{\begin{aligned} X_t(t) &= -\kp X(t) + f_1(X(t), Y(t)) + e^{-t} g_1(X(t), Y(t)) \\
Y_t(t) &= \mu Y(t) + f_2(X(t), Y(t)) + e^{-t} g_2(X(t), Y(t)) \end{aligned}\right.,$$ where $\mu > 0 > - \kp$, $\frac{f_i(X, Y)}{\sqrt{X^2 + Y^2}} \rightarrow 0$ as $(X, Y) \rightarrow (0, 0)$ and $g_i \in O(\sqrt{X^2 + Y^2})$ as $(X, Y) \rightarrow (0, 0)$ $\forall \, i = 1, 2$. Then for all positive $M_1, M_2$ and $\dt \in (0, \kp)$, there exist positive constants $\vep_0$ and $T_0$ with the following significance: for any $\vep \in (0, \vep_0]$ and $T \in [T_0, \infty)$, there is a solution curve $(X(t), Y(t))$ defined on $t \in [T, \infty)$ with initial condition $X(T) = \vep$ and $(X, Y) \rightarrow (0, 0)$ as $t \rightarrow \infty$. Furthermore, $$\left\{\begin{aligned}
    0 \leq X &\leq \vep \\
    -M_2 X \leq Y &\leq M_1 X
\end{aligned}\right. \quad \text{for all } t \in [T, \infty)$$ and $X, Y \in O(e^{- (\kp-\dt)t})$ as $t \rightarrow \infty$.
\end{thm}

\begin{proof}
Given any $M_1, M_2 > 0$ and $\dt \in (0, \kp)$, there exist positive constants $\vep_0$ and $T_0$ such that Lemma \ref{lem1} and Lemma \ref{lem2} hold for all $\vep \in (0, \vep_0]$ and $T \in [T_0, \infty)$. In other words, fixing $\vep \in (0, \vep_0]$ and $T \in [T_0, \infty)$, Lemma \ref{lem1} implies that those solutions with initial conditions $(X(T), Y(T)) \in E_{M_1, M_2, \vep}^+$ strictly decrease in the $X$-direction when they remain in $R_{M_1, M_2, \vep}^+$. In particular, such solutions can remain in $R_{M_1, M_2, \vep}^+$ for all $t > T$ only if they tend to $(0, 0)$. 

On the other hand, Lemma \ref{lem2} shows that there is a set of initial conditions $\{(X(T), Y(T))\} \subset E_{M_1, M_2, \vep}^+$ with solutions that eventually exit $R_{M_1, M_2, \vep}^+$ to the top. Similarly, there also exists another set of initial conditions $\{(X(T), Y(T))\} \subset E_{M_1, M_2, \vep}^+$ with solutions that eventually exit $R_{M_1, M_2, \vep}^+$ to the below. Due to the smooth dependence of initial conditions, these two sets are both open intervals. Note that $E_{M_1, M_2, \vep}^+$ is connected. We therefore conclude that there exists a nonempty set of initial conditions $\{(X(T), Y(T))\} \subset E_{M_1, M_2, \vep}^+$ such that the solutions never leave $R_{M_1, M_2, \vep}^+$. That is to say, the solutions tend to $(0, 0)$ as $t \rightarrow \infty$.

Moreover, whenever $t \geq T$, since $X_t \leq - (\kp-\dt) X$ by Lemma \ref{lem1}, the Grönwall's inequality shows that \begin{equation} \label{cond3}
    X(t) \leq X(T) e^{- (\kp-\dt) (t - T)} = \tilde{C} e^{- (\kp-\dt) t},
\end{equation} where $\tilde{C} = \vep e^{(\kp-\dt)T}$ is a positive constant. That is to say, $X \in O(e^{ - (\kp-\dt) t})$ as $t \rightarrow \infty$. It follows from $- M_2 X \leq Y \leq M_1 X$ that $Y \in O(e^{ - (\kp-\dt) t})$ as $t \rightarrow \infty$.
\end{proof}

\begin{rmk}
The aforementioned arguments also work on $$R_{M_1, M_2, \vep}^- \coloneqq \{(X, Y) \in S_\vep \, \big| \, M_1 X \leq Y \leq - M_2 X, X \leq 0\}$$ and $$E_{M_1, M_2, \vep}^- \coloneqq R_{M_1, M_2, \vep}^- \cap \{X = -\vep\}.$$ In other words, there is also a stable curve in the half plane $\{X \leq 0\}$.
\end{rmk}


\section{\textbf{The existence of self-expanders}}

Given a constant $C = 1 \text{ or } -1$, the self-similar solution ($C = 1$ is self-expander and $C = -1$ is self-shrinker) is characterized by the equation \eqref{eq1}). Let $t \coloneqq \log r$, $\vph \coloneqq \frac{f}{r}$ and $\psi \coloneqq \vph_t$. Then $f_r = \vph + \psi$ and $f_{rr} = \frac{1}{r}(\psi_t + \psi) = e^{-t} (\psi_t + \psi)$. We can therefore convert the second order ODE to the following system of first order ODEs: \begin{equation} \label{eq3}
    \left\{\begin{aligned} \vph_t &= \psi \\
\psi_t &= -\psi - ((n - p + \frac{p}{1 + \ld^2 \vph^2} + C e^{2t})\psi + (n - p + \frac{(1- \ld^2)p}{1 + \ld^2 \vph^2})\vph)(1+(\vph+\psi)^2) \end{aligned}\right..
\end{equation}

Note that the system \eqref{eq3} has exactly two equilibria $(0, 0)$ and $(\vph_0, 0)$, where $\vph_0 = \sqrt{\frac{p \ld^2 -n }{(n - p) \ld^2}}$, in the half plane $\{\vph \geq 0\}$. If we ignore the exponential term, considering the remaining autonomous system, then at $(0, 0)$, the linearized system looks like $$\begin{pmatrix}
0 & 1 \\
p \ld^2 - n & -n -1 \end{pmatrix} = \begin{pmatrix}
0 & 1 \\
k(n+k-1) - n & -n -1 \end{pmatrix}$$ with eigenvalues $\tilde{\ld}_1 = k - 1, \tilde{\ld}_2 = - n - k$ and associated eigenvector $V_1 = \begin{pmatrix}
1 \\
k-1 \end{pmatrix}, V_2 = \begin{pmatrix}
1 \\
-n-k \end{pmatrix}$. It implies that $(0, 0)$ is a saddle. At $(\vph_0, 0)$, the linearized system looks like $$\begin{pmatrix}
0 & 1 \\
2n(\frac{n}{p\ld^2} - 1) & -n -1 \end{pmatrix} = \begin{pmatrix}
0 & 1 \\
2n(\frac{n}{k(n+k-1)} - 1) & -n -1 \end{pmatrix}$$ with eigenvalues $\tilde{\ld} = -\frac{n+1}{2} \pm \frac{1}{2} \sqrt{n^2 - 6n + 1 + \frac{8 n^2}{k(k+n-1)}}$. When $n = 3, k \geq 4$ or $n = 5, k \geq 6$, $\tilde{\ld}$ are complex numbers with negative real parts. Hence, \begin{enumerate}
    \item If $(n, p, k) = (3, 2, 2), (5, 4, 2), (5, 4, 4)$ or $n \geq 7$, then $(\vph_0, 0)$ is a sink.
    \item If $(n, p) = (3, 2), k \geq 4$ or $(n, p) = (3, 2), k \geq 6$, then $(\vph_0, 0)$ is a spiral sink.
\end{enumerate}

Here, we aim to construct a compact positive invariant set in the first quadrant.

\begin{prop}
\label{prop2}
 For the self-expander case, i.e. $C = 1$, we have the following:
 \begin{enumerate}
     \item If $(n, p, k) = (3, 2, 2), (5, 4, 2), (5, 4, 4)$, then the compact region $\Dt$ enclosed by $\{\psi = 0\}$ and the graph of $$g(\vph) = \frac{3(\frac{(\ld^2 - 1)p}{1+\ld^2 \vph^2} - (n - p))\vph}{2(n - p)}$$ is a positive invariant set of the system \eqref{eq3}.
     \item If $n \geq 7$, then the compact region $\Dt$ enclosed by $\{\psi = 0\}$ and the graph of $$g(\vph) = \frac{2(\frac{(\ld^2 - 1)p}{1+\ld^2 \vph^2} - (n - p))\vph}{(n - p)}$$ is a positive invariant set of the system \eqref{eq3}.
 \end{enumerate}
\end{prop}

\begin{proof}
It suffices to check the following two conditions:
\begin{enumerate} 
\item $\psi_t \geq 0$ on $\{ (\vph, 0) \, \big| \, 0 \leq \vph \leq \vph_0\}$.
\item $\langle (\vph_t, \psi_t), ( g'(\vph), - 1) \rangle \geq 0$ on $\{ (\vph, g(\vph)) \, \big| \, 0 \leq \vph \leq \vph_0\}$, where $( g'(\vph), - 1)$ is the inner normal of the graph of $g$ and $\langle \cdot, \cdot \rangle$ denotes the standard inner product on $\BR^2$.
\end{enumerate}

The first one is clear, and the second one is typically based on the results by Xu, Yang and Zhang (cf. \cite[Sec. 4.3]{XYZ19}). For each case, they have proved that $$g'(\vph) > \frac{X_2}{X_1}(\vph, g(\vph)) \quad \forall \, \vph \in (0, \vph),$$ where \begin{equation}
    \left\{\begin{aligned} X_1(\vph, \psi) &= \psi \\
X_2(\vph, \psi) &= -\psi - ((n - p + \frac{p}{1 + \ld^2 \vph^2})\psi + (n - p + \frac{(1- \ld^2)p}{1 + \ld^2 \vph^2})\vph)(1+(\vph+\psi)^2) \end{aligned}\right..
\end{equation}
Then condition (2) follows from direct calculations: $$\begin{aligned}\langle (\vph_t, \psi_t), ( g'(\vph), - 1) \rangle &= g'(\vph) \psi + \psi + (1+(\vph+\psi)^2)\\
&\quad \cdot ((n - p + \frac{p}{1 + \ld^2 \vph^2} + e^{2t})\psi + (n - p + \frac{(1- \ld^2)p}{1 + \ld^2 \vph^2})\vph) \\
&\geq g'(\vph) \psi + \psi + (1+(\vph+\psi)^2) \\
&\quad \cdot ((n - p + \frac{p}{1 + \ld^2 \vph^2})\psi + (n - p + \frac{(1- \ld^2)p}{1 + \ld^2 \vph^2})\vph) \\
&= g'(\vph) X_1 - X_2 \\
&\geq 0
\end{aligned}$$ on $\{ (\vph, g(\vph)) \, \big| \, 0 \leq \vph \leq \vph_0\}.$ It finishes the proof of this proposition.
\end{proof}

\begin{rmk}
The tricky point is that we need to find a function $g$ with $$g'(0) > k - 1 > 0$$ in order to apply Theorem \ref{thm1} near $(0, 0)$ in $\Dt$. In this case, \begin{equation} \label{slope}
    g'(0) = \left\{\begin{aligned}
    &\frac{3(n-1)}{2(n-p)}(k-1) \quad \text{if } (n, p, k) = (3, 2, 2), (5, 4, 2), (5, 4, 4) \\
    &\frac{2(n-1)}{n-p}(k-1) \quad \text{if } n \geq 7
    \end{aligned}\right..
\end{equation}
\end{rmk}

\begin{proof}[Proof of Main Theorem and properties \eqref{result1} and \eqref{result2}]
From now on, we consider only $(n, p, k) = (3, 2, 2)$, $(5, 4, 2)$, $(5, 4, 4)$ or $n \geq 7$ and the self-expander case $C=1$.

Let $X(t) \coloneqq \frac{n + k}{n + 2k - 1} \vph(-t) + \frac{1}{n + 2k - 1}\psi(-t)$, $Y(t) \coloneqq \frac{k - 1}{n + 2k - 1} \vph(-t) - \frac{1}{n + 2k - 1} \psi(-t)$. Then the system \eqref{eq3} changes into the form $$\left\{\begin{aligned} 
X_t &= (1 - k) X + O(X^2 + Y^2) + e^{-2t} \cdot O(\sqrt{X^2 + Y^2}) \\
Y_t &= (n + k) Y + O(X^2 + Y^2) + e^{-2t} \cdot O(\sqrt{X^2 + Y^2})
 \end{aligned}\right.,$$ which satisfies all the assumptions in Theorem \ref{thm1}.
 
Hence, we can simply choose $\dt = \frac{1}{2}$, $M_1 > 0$ and $0 < M_2 << 1$ so that $\{Y = M_1 X\} = \{\psi = 0\}$ and $\{Y = - M_2 X\} = \{\psi = 2(k-1) \vph\}$. After applying Theorem \ref{thm1}, there exist positive constants $\hat{\vep_0}$ and $T_0$ with the following properties: for any $\hat{\vep} \in (0, \hat{\vep_0}]$ and $T \in [T_0, \infty)$, there is a solution curve $(X(t), Y(t))$ defined on $t \in [T, \infty)$ with initial condition $X(T) = \hat{\vep}$ and $(X, Y) \rightarrow (0, 0)$ as $t \rightarrow \infty$. Furthermore, $$\left\{\begin{aligned}
    0 \leq X &\leq \hat{\vep} \\
    -M_2 X \leq Y &\leq M_1 X
\end{aligned}\right. \quad \text{for all } t \in [T, \infty)$$ and $X, Y \in O(e^{- (k - \frac{3}{2})t})$ as $t \rightarrow \infty$.

Since $\vph(t) = X(-t) + Y(-t), \psi(t) = (k-1) X(-t) - (n+k) Y(-t)$, we conclude that there exist positive constants \begin{equation} \label{const}
    \vep_0 << 1 \text{ and } T_0 > \frac{2}{2k-3} \log \ld > 0
\end{equation} with the following significance: for any $\vep \in (0, \vep_0]$ and $-T \in (- \infty, - T_0]$, there is a solution curve $(\vph(t), \psi(t))$ defined on $t \in (- \infty, -T]$ such that $\vph(-T) = \vep$ and $(\vph, \psi) \rightarrow (0, 0)$ as $t \rightarrow -\infty$. Moreover, \begin{equation} \label{cond2}
    \left\{\begin{aligned}
        0 \leq \vph &\leq \vep \\
        0 \leq \psi &\leq 2(k-1) \vph
    \end{aligned}\right. \quad \text{for all } t \in (-\infty, -T]\end{equation}
and $\vph, \psi \in O(e^{(k - \frac{3}{2})t})$ as $t \rightarrow -\infty$. Note that due to \eqref{slope}, \eqref{cond2} implies that $(\vph, \psi) \in \Dt \, \forall \, t \in (-\infty, -T]$.

Now, Proposition \ref{prop2} implies that $\Dt$ is a compact positive invariant set. It follows that we can actually extend $(\vph, \psi)$ to be a global solution (cf. \cite[Chap. 7.2, p.146--147]{HSD13}). That is to say, there is a solution curve $(\vph(t), \psi(t))$ defined on $t \in (-\infty, \infty)$ such that $\vph(-T) = \vep$ and $(\vph, \psi) \in \Dt \, \forall \, t$.

Recall that $t = \log r$, $f = r \vph$ and $f_r = \vph + \psi$. Then property \eqref{result1} follows from $(\vph, \psi) \in \Dt$. Moreover, since $\vph, \psi \in O(e^{(k - \frac{3}{2})t})$ as $t \rightarrow -\infty$, we have that $$\begin{aligned}
    f \in O(r^{k - \frac{1}{2}}) \;\text{ and } \;
    f_r \in O(r^{k - \frac{3}{2}})
\end{aligned}$$ as $r \rightarrow 0$. Therefore, 
$$F(r, \bx) = (r \bx, f(r) \CL(\bx))$$ is $\CC^1$ near the origin. Applying the classical bootstrapping argument (cf. \cite[Theorem 6.8.1]{Mor66} or \cite[Theorem 9.13]{GT01}), which follows from the elliptic regularity and Sobolev embedding, $F$ is actually analytic near $r = 0$. Combing with the fact that $k$ is an integer, we further conclude that $$\begin{aligned}
    f \in O(r^{k}) \;\text{ and } \; f_r \in O(r^{k - 1})
\end{aligned}$$ as $r \rightarrow 0$. This finishes the proof of property \eqref{result2}.

Note that $f(r)$ is smooth for all $r > 0$. It follows that $F(r, \bx)$ is smooth everywhere. In other words, $$\Sm = F(\BR_{\geq 0} \times \BS^n) = \{(r \bx, f(r) \CL(\bx)) \, \big| \, r \in \BR_{\geq 0}, \bx \in \BS^{n}\}$$ is a smooth self-expander.
\end{proof}


\section{\textbf{The uniqueness of self-expanders under the assumption (\ref{cond1})}}

After letting $t \coloneqq \log r$, $\vph \coloneqq \frac{f}{r}$ and $\psi \coloneqq \vph_t$ as with section 4, we see that the assumption \eqref{cond1} is equivalent to \eqref{cond2}. In other words, if a smooth self-expander satisfies the assumption \eqref{cond1}, then it must be constructed by using Theorem \ref{thm1}.

In this regard, we have the following uniqueness property.

\begin{prop}
Let $\vep_0$, $T_0$ be the constants in \eqref{const} and $R_0 \coloneqq e^{-T_0}$. Given $\vep \in (0, \vep_0]$ and $R \coloneqq e^{-T} \in (0, R_0]$, then there is at most one $f$ defined on $[0, \infty)$ satisfying the self-expander ($C = 1$) ODE \eqref{eq1}, with $f(R) = \vep R$, $f(0) = 0$ and the assumption \eqref{cond1}.
\end{prop}

\begin{proof}
We first show that there is at most one $f$ defined on $[0, R]$ satisfying \eqref{eq1}, $f(R) = \vep R$, $f(0) = 0$ and \eqref{cond1}. Suppose that both $f_1, f_0$ defined on $[0, R]$ satisfy all the conditions. Let $g(r) = f_1(r) - f_0(r)$. We first notice that $g$ is continuous on $[0, R]$. Moreover, it is at least $\CC^2$, in fact smooth, on $(0, R)$.

We compute that $$\begin{aligned}
    \frac{(f_1)_{rr}}{1+(f_1)_{r}^2} - \frac{(f_0)_{rr}}{1+(f_0)_{r}^2} =& \frac{g_{rr}(1+(f_0)_{r}^2) - g_{r} (f_0)_{rr} ((f_0)_{r} + (f_1)_{r})}{(1+(f_1)_{r}^2)(1+(f_0)_{r}^2)}, \\
(n-p) \frac{(f_1)_{r}}{r} - (n-p) \frac{(f_0)_{r}}{r} =& \, (n-p)\frac{g_{r}}{r}, \\
(r(f_1)_{r} - f_1) - (r(f_0)_{r} - f_0) =& \, rg_{r} - g, \\
\frac{p(r (f_1)_{r} - \ld^2 f_1)}{r^2+ \ld^2 f_1^2} - \frac{p(r (f_0)_{r} - \ld^2 f_0)}{r^2 + \ld^2 f_0^2} =& \, \frac{p g_{r}(r^3 + \ld^2 r f_0^2)}{(r^2+ \ld^2 f_1^2)(r^2 + \ld^2 f_0^2)} \\
&-\frac{p g ( \ld^2 r^2 - \ld^4 f_0 f_1 + \ld^2 r (f_0)_{r} (f_0+f_1))}{(r^2+ \ld^2 f_1^2)(r^2 + \ld^2 f_0^2)}.
\end{aligned}$$
Therefore, $g$ satisfies the following ODE: \begin{equation} \label{eq4}
    \begin{aligned}
    &\frac{g_{rr}(1+(f_0)_{r}^2) - g_{r} (f_0)_{rr} ((f_0)_{r} + (f_1)_{r})}{(1+(f_1)_{r}^2)(1+(f_0)_{r}^2)} + \frac{ (n-p) g_{r}}{r} + rg_{r} - g\\
    +& \frac{p (g_{r}(r^3 + \ld^2 r f_0^2)- g ( \ld^2 r^2 - \ld^4 f_0 f_1 + \ld^2 r (f_0)_{r} (f_0+f_1)))}{(r^2+ \ld^2 f_1^2)(r^2 + \ld^2 f_0^2)} = 0.\end{aligned}
\end{equation}

Suppose that $r_1 \in (0, R)$ is a local maximum point of $g$. Then $g_{r} (r_1) = 0$ and $g_{rr}(r_1) \leq 0$. At $r_1$, \eqref{eq4} becomes \begin{equation} \label{eq5}
    \frac{1+(f_0)_{r}^2}{(1+(f_1)_{r}^2)(1+(f_0)_{r}^2)}g_{rr} - (\frac{p ( \ld^2 r_1^2 - \ld^4 f_0 f_1 + \ld^2 r_1 (f_0)_{r} (f_1+f_0))}{(r_1^2+\ld^2f_1^2)(r_1^2+\ld^2f_0^2)} + 1)g = 0.
\end{equation}

According to the inequality \eqref{cond3} in the proof of Theorem \ref{thm1} and \eqref{const}, we see that $$f_i(r) \leq \vep R^{k - \frac{3}{2}} r^{k - \frac{1}{2}} \leq \ld^{-1} r^{k - \frac{1}{2}} \quad \forall \, i = 1, 2 \text{ and } r \in (0, R).$$ Then $\ld^2 r_1^2 - \ld^4 f_0(r_1) f_1(r_1) \geq \ld^2 r_1^2 (1 - r_1^{2k - 3})\geq 0$. Combining with $f_0, f_1, (f_0)_{r}, (f_1)_{r} \geq 0$, which follows from the assumption \eqref{cond1}, we conclude that $$\frac{1+(f_0)_{r}^2}{(1+(f_1)_{r}^2)(1+(f_0)_{r}^2)} \geq 0 \text{ and } \frac{p ( \ld^2 r_1^2 - \ld^4 f_0 f_1 + \ld^2 r_1 (f_0)_{r} (f_1+f_0))}{(r_1^2+\ld^2f_1^2)(r_1^2+\ld^2f_0^2)} + 1 \geq 0.$$ It follows from the ODE \eqref{eq5} that $g(r_1)$ must not exceed $0$.

If $r_2 \in (0, R)$ is a local minimum point of $g$, then a similar argument shows that $g(r_2)$ cannot be less than $0$. Combining with $g(0) = g(R) = 0$, we conclude that $g \equiv 0$. That is to say, $f_1 \equiv f_0$.

Finally, the uniqueness of the extension of $f$ on $[R, \infty)$ follows from the Picard-Lindelöf theorem. Hence, at most one $f$ defined on $[0, \infty)$ satisfying \eqref{eq1}, $f(R) = \vep R$, $f(0) = 0$ and \eqref{cond1}.
\end{proof}


\section{\textbf{The behavior of the self-expander at infinity}}

In this section, we investigate the behavior of self-expander we construct at infinity. We first go back to the system \eqref{eq3} (C = 1).

Note that in the positive invariant set $\Dt$ defined in Proposition \ref{prop2}, $\vph, \psi \geq 0$. Then we have $$\psi_t \leq 0$$ or $$(n - p + \frac{p}{1 + \ld^2 \vph^2} + e^{2t})\psi + (n - p + \frac{(1- \ld^2)p}{1 + \ld^2 \vph^2})\vph < 0.$$ In other words, $$e^{2t} \psi < \vph((\ld^2 p - n) - \ld^2 (n-p)\vph^2).$$ The critical point of the right hand side above is $\vph = \pm \sqrt{\frac{\ld^2 p - n}{3 \ld^2 (n-p)}} = \pm \frac{\vph_0}{\sqrt{3}}$. Therefore, $$\vph((\ld^2 p - n) - \ld^2 (n-p)\vph^2) \leq \frac{\vph_0}{\sqrt{3}} ((\ld^2 p - n) - \ld^2 (n-p) \frac{\vph_0^2}{3}) = \frac{2\ld^2 (n-p) \vph_0^3}{3 \sqrt{3}}.$$ We conclude that \begin{equation} \label{psi_t}
\psi_t > 0 \text{ only if } \psi < \Tilde{C} e^{-2t},\end{equation} where $\Tilde{C} = \frac{2 \ld^2 (n-p) \vph_0^3}{3 \sqrt{3}}.$

Now, we observe that $\lim_{t \rightarrow \infty} \vph$ exists since $\Dt$ is compact and $\vph_t = \psi \geq 0$ in $\Dt$.
\begin{prop}
$\lim_{t \rightarrow \infty} \psi$ also exists and is equal to $0$.
\end{prop}
\begin{proof}
We split it into two cases.
\begin{enumerate}
    \item Suppose that $\exists \, T > 0$ such that $\psi(t) \neq \Tilde{C} e^{-2t}$ $\forall \, t > T$. Therefore, either \begin{enumerate}
        \item $\psi(t) > \Tilde{C} e^{-2t}$ $\forall \, t > T$ or
        \item $\psi(t) < \Tilde{C} e^{-2t}$ $\forall \, t > T$ 
    \end{enumerate} happens.
    
    For (a), note that $\psi_t(t) < 0$ $\forall \, t > T$. Since $\Dt$ is compact, it implies that $\lim_{t \rightarrow \infty} \psi$ exists. Moreover, since $\lim_{t \rightarrow \infty} \vph$ exists and $\vph_t = \psi$, $\lim_{t \rightarrow \infty} \psi$ must equal $0$.
    
    For (b), note that $0 \leq \psi(t) < \Tilde{C} e^{-2t}$ $\forall \, t > T$. Then by the squeeze lemma, $$\lim_{t \rightarrow \infty} \psi = 0.$$
    
    \item Suppose that $\forall \, \Tilde{T} > 0$, $\exists \, T > \Tilde{T}$ such that $\psi(T) = \Tilde{C} e^{-2T}$. We first claim that if $\psi(T) = \Tilde{C} e^{-2T}$, then $\psi(t) < \Tilde{C} e^{-2T}$ $\forall \, t > T$. We argue it by contradiction.
    
    Assume that the statement is false, say $\exists \, t_1 > T$ such that $\psi(t) > \Tilde{C} e^{-2T} > \Tilde{C} e^{-2t_1}$. Let $$g(t) \coloneqq \psi(t) - \Tilde{C} e^{-2t}.$$ Define $$S \coloneqq \{ t \in [T, t_1] \, \big| \, g(t) = 0\}.$$ Since $S$ is bounded, $\sup S$ exists. Moreover, the continuity of $F$ and the fact that $g(t_1) > 0$ show that $t_0 \coloneqq \sup S < t_1$. Now, by the Intermediate Value Theorem, $$\psi(t) > \Tilde{C} e^{-2t} \quad \forall \, t \in (t_0, t_1].$$ Furthermore, by the Mean Value Theorem, $\exists \, t_2 \in (t_0, t_1)$ such that $$\psi_t(t_2) = \frac{\psi(t_1) - \psi(t_0)}{t_1 - t_0} > 0,$$ which contradicts to \eqref{psi_t}.
    
    Finally, according to the claim, we have a sequence of $\{T_i\}_{i = 1}^{\infty}$ such that $T_i < T_j$ $\forall \, i < j$, $T_i \rightarrow \infty $ as $i \rightarrow \infty$ and $\psi(t) < \Tilde{C} e^{-2T_i}$ $\forall \, t > T_i$. By the squeeze lemma, we conclude that $$\lim_{t \rightarrow \infty} \psi = 0.$$
\end{enumerate}
\end{proof}

Let $\vph_{\infty} \coloneqq \lim_{t \rightarrow \infty} \vph$. Recall that $f = r \vph$, then the aforementioned discussion shows that the self-expander $$\Sigma = \{F(r, \bx) = ( r \bx, f(r) \CL(\bx) ) \, \big| r \in \BR_{\geq 0}, \bx \in \BS^{n}\}$$ is asymptotic to the cone $$\{ ( r \bx, r \vph_{\infty} \CL(\bx)) \, \big| \, r \in \BR_{\geq 0}, \bx \in \BS^{n}\}$$ as $r \rightarrow \infty$.


\bibliographystyle{abbrv}



\end{document}